\documentclass{amsart}
\usepackage{amsfonts,amssymb,latexsym,amsmath,verbatim,url,color,tikz,mathrsfs}
\usepackage[margin=1.5in]{geometry}
\usepackage[pdfpagelabels]{hyperref}
\usepackage[nameinlink]{cleveref}
\usetikzlibrary{matrix,arrows}

\def\N{\mathbb{N}}
\def\Z{\mathbb{Z}}
\def\F{\mathbb{F}}
\def\A{\mathbb{A}}
\def\P{\mathbb{P}}
\def\G{\mathbb{G}}
\def\H{\mathrm{H}}
\def\Q{\mathbb{Q}}
\DeclareMathOperator{\Proj}{Proj}
\DeclareMathOperator{\Spec}{Spec}
\DeclareMathOperator{\Br}{Br}
\DeclareMathOperator{\Pic}{Pic}
\DeclareMathOperator{\GL}{GL}

\DeclareMathOperator{\Mat}{Mat}
\DeclareMathOperator{\id}{id}
\DeclareMathOperator{\im}{im}
\DeclareMathOperator{\lcm}{lcm}
\DeclareMathOperator{\Hom}{Hom}
\DeclareMathOperator{\Mor}{Mor}
\DeclareMathOperator{\tors}{tors}
\DeclareMathOperator{\Sch}{Sch}
\DeclareMathOperator{\op}{op}
\DeclareMathOperator{\rank}{rank}
\DeclareMathOperator{\Aut}{Aut}
\DeclareMathOperator{\Sym}{Sym}

\newtheoremstyle{mystyle}
  {}
  {}
  {}
  {0pt}
  {}
  {.}
  { }
  {\textbf{\thmname{#1}\thmnumber{ #2}}\thmnote{ (#3)}}

\theoremstyle{mystyle}
\numberwithin{equation}{subsection}
\newtheorem{theorem}[subsection]{Theorem}
\newtheorem{lemma}[subsection]{Lemma}
\newtheorem{remark}[subsection]{Remark}
\newtheorem{definition}[subsection]{Definition}

\newtheoremstyle{pgstyle} {} {} {} {} {} {.} { } {\textbf{\thmname{#1}\thmnumber{#2}}\thmnote{ (#3)}}
\theoremstyle{pgstyle}
\newtheorem{pg}[subsection]{}

\newcommand{\mf}[1]{\mathfrak{#1}}
\newcommand{\ms}[1]{\mathscr{#1}}
\newcommand{\mc}[1]{\mathcal{#1}}
\newcommand{\mb}[1]{\mathbf{#1}}
\newcommand{\mr}[1]{\mathrm{#1}}
\newcommand{\ml}[1]{\mathsf{#1}}

\newcommand{\et}{\operatorname{\acute et}}
\newcommand{\zar}{\operatorname{zar}}
\newcommand{\fppf}{\operatorname{fppf}}
\newcommand{\BM}[1]{\begin{bmatrix} #1 \end{bmatrix}}

\begin{document}
\title{The cohomological Brauer group of weighted projective spaces and stacks}
\author{Minseon Shin}
\email{shinms@uw.edu}
\urladdr{\url{http://sites.math.washington.edu/~shinms}}
\address{Department of Mathematics \\ University of Washington \\ Seattle, WA 98195-4350 USA}
\begin{abstract} We compute the cohomological Brauer groups of twists of weighted projective spaces and weighted projective stacks. \end{abstract}
\date{\today}
\maketitle

\section{Introduction}

For any scheme $S$, we denote $\Br'(S) := \H_{\et}^{2}(S,\G_{m})_{\tors}$ the \emph{cohomological Brauer group} of $S$. In this paper, we are interested in the cohomological Brauer groups of twists of weighted projective spaces and weighted projective stacks.

Let $n \ge 1$ and let $\rho = (\rho_{0},\dotsc,\rho_{n})$ be an $(n+1)$-tuple of positive integers. There is a $\G_{m}$-action on $\A^{n+1}$ sending $u \cdot (t_{0},\dotsc,t_{n}) \mapsto (u^{\rho_{0}}t_{0},\dotsc,u^{\rho_{n}}t_{n})$; let $\mc{P}_{\Z}(\rho) := [(\A_{\Z}^{n+1} \setminus \{0\}) / \G_{m}]$ denote the quotient stack; for any scheme $S$, the base change $\mc{P}_{S}(\rho) := \mc{P}_{\Z}(\rho) \times_{\Spec \Z} S$ is called the \emph{weighted projective stack} associated to $\rho$ over $S$. \par The assumption that each $\rho_{i}$ is positive implies that the inertia stack of $\mc{P}_{\Z}(\rho)$ is finite; hence $\mc{P}_{\Z}(\rho)$ admits a coarse moduli space which may be described as follows \cite[\S2.1]{ABRAMOVICHHASSETT-SVWAT}. We view the polynomial ring $R := \Z[t_{0},\dotsc,t_{n}]$ as a $\Z$-graded ring where $\deg(t_{i}) = \rho_{i}$, and set $\P_{\Z}(\rho) := \Proj R$ and $\P_{S}(\rho) := \P_{\Z}(\rho) \times_{\Spec \Z} S$ for any scheme $S$. The scheme $\P_{S}(\rho)$ is the \emph{weighted projective space} associated to $\rho$ over $S$. Weighted projective spaces often arise in the construction of moduli spaces. For example, we may naturally view the moduli space of elliptic curves (in characteristic $0$) as an open subscheme of $\P(4,6)$. The moduli space of cubic surfaces is isomorphic to $\P(1,2,3,4,5)$, see \cite{REINECKE-MSOCS2012}.

\begin{theorem} \label{0001} Let $f_{X} : X \to S$ be a morphism of schemes such that there exists an etale surjection $S' \to S$ such that $X \times_{S} S' \simeq \P_{S'}(\rho)$. Then there is an exact sequence \begin{align} \label{0001-01} \Gamma(S,\underline{\Z}) \to \Br'(S) \stackrel{f_{X}^{\ast}}{\to} \Br'(X) \to 0 \end{align} of groups. \end{theorem}

\begin{theorem} \label{0021} Let $S$ be a scheme, let $f_{\mc{X}} : \mc{X} \to S$ be a morphism of algebraic stacks such that there exists an etale surjection $S' \to S$ such that $\mc{X} \times_{S} S' \simeq \mc{P}_{S'}(\rho)$. Then there is an exact sequence \begin{align} \label{0021-01} \Gamma(S,\underline{\Z}) \to \Br'(S) \stackrel{f_{\mc{X}}^{\ast}}{\to} \Br'(\mc{X}) \to 0 \end{align} of groups. Let $\pi : \mc{X} \to X$ be the coarse moduli space of $\mc{X}$. Then $X$ satisfies the hypothesis of \Cref{0001}, and \labelcref{0001-01} and \labelcref{0021-01} fit into a commutative diagram \begin{equation} \label{0021-02} \begin{tikzpicture}[>=angle 90, baseline=(current bounding box.center)] 
\matrix[matrix of math nodes,row sep=3em, column sep=2em, text height=1.9ex, text depth=0.5ex] { 
|[name=11]| \Gamma(S,\underline{\Z}) & |[name=12]| \Br'(S) & |[name=13]| \Br'(\mc{X}) & |[name=14]| 0 \\ 
|[name=21]| \Gamma(S,\underline{\Z}) & |[name=22]| \Br'(S) & |[name=23]| \Br'(X) & |[name=24]| 0 \\ 
}; 
\draw[-,font=\scriptsize,transform canvas={xshift= 1pt}](12) edge (22);
\draw[-,font=\scriptsize,transform canvas={xshift=-1pt}](12) edge (22);
\draw[->,font=\scriptsize]
(11) edge (12) (12) edge node[above=0pt] {$f_{\mc{X}}^{\ast}$} (13) (13) edge (14)
(21) edge (22) (22) edge node[below=0pt] {$f_{X}^{\ast}$} (23) (23) edge (24)
(21) edge node[left=0pt] {$\times \lcm(\rho)$} (11) (23) edge node[right=0pt] {$\pi^{\ast}$} (13); \end{tikzpicture} \end{equation} where the leftmost vertical map denotes multiplication-by-$\lcm(\rho)$. \end{theorem}

As we recall in \Cref{0004}, a weighted projective space $\P(\rho)$ is a toric variety. In \Cref{sec-02}, we use the results of DeMeyer, Ford, Miranda \cite{DEMEYERFORDMIRANDA-TCBGOATV} on the Brauer group of toric varieties to compute the Brauer group of $\P(\rho)$ over an algebraically closed field; taking the prime-to-$p$ limit of dilations of the toric variety reduces us to computing the $p$-torsion when each weight $\rho_{i}$ is a power of $p$. A deformation theory argument of Mathur \cite{MATHUR-EC20190726}, which uses a Tannaka duality result of Hall, Rydh \cite{HALLRYDH-CTDAAOHS2014}, allows us to deduce \Cref{0001} over a strictly henselian local ring (see \Cref{20190128-05}); then a Leray spectral sequence argument \Cref{0009} gives the result for arbitrary schemes.

The proof of \Cref{0021} is analogous to that of \Cref{0001}, except in case $S$ is the spectrum of a field, which we deduce using that the $\G_{m}$-action on $\A^{n+1}$ extends to an action of the multiplicative monoid $\A^{1}$ on $\A^{n+1}$.

\begin{pg}[Acknowledgements] I am grateful to Tim Ford, Daniel Huybrechts, Max Lieblich, and Siddharth Mathur for helpful conversations. I also thank the Max Planck Institute for Mathematics in Bonn for their hospitality and financial support. \end{pg}

\section{Weighted projective spaces}

For general background on weighted projective spaces, see e.g. \cite{DOLGACHEV-WPV1982}, \cite{RT-WPSFTTPOVWCA}. The $k$-rational points of $\P_{k}(\rho)$ may be viewed as the quotient of $k^{n+1} \setminus \{(0,\dotsc,0)\}$ by the equivalence relation $(x_{0},\dotsc,x_{n}) \sim (u^{\rho_{0}}x_{0},\dotsc,u^{\rho_{n}}x_{n})$; thus weighted projective spaces are also called \emph{twisted projective spaces}.

The weighted projective space $\P_{\Z}(\rho)$ is projective \cite[II, (2.1.6), (2.4.7)]{EGA}. It is not necessarily true that the $(\lcm(\rho))$th Veronese subring of $\Z[t_{0},\dotsc,t_{n}]$ is generated in degree 1; for example, if $\rho = (1,6,10,15)$, the monomial $t_{0}^{1}t_{1}^{4}t_{2}^{2}t_{3}^{1}$ has degree $60 = 2\lcm(\rho)$ but is not a product of two monomials of degree $30$ \cite[2.6]{DELORME-ESPACESPROJECTIFSANISOTROPES}.

\begin{pg} \label{0022} Suppose a positive integer $d$ divides all $\rho_{i}$ and set $\rho/d := (\rho_{0}/d,\dotsc,\rho_{n}/d)$. Then $\P_{\Z}(\rho) \simeq \P_{\Z}(\rho/d)$ by \cite[II, (2.4.7) (i)]{EGA}, and under this isomorphism $\mc{O}_{\P_{\Z}(\rho/d)}(\ell)$ corresponds to $\mc{O}_{\P_{\Z}(\rho)}(d\ell)$ for all $\ell \in \Z$. \end{pg}

By \Cref{0005}, we may also assume that $\gcd(\{\rho_{j}\}_{j \ne i}) = 1$ for all $i$:

\begin{lemma}[Reduction of weights] \label{0005} \cite[1.3]{DELORME-ESPACESPROJECTIFSANISOTROPES}, \cite[1.3.1]{DOLGACHEV-WPV1982}, \cite[1.3, 1.4]{ALAMRANI-CDIEGDPDFPT} Suppose $\gcd(\rho) = 1$ and set \begin{align*} d_{i} &:= \gcd(\{\rho_{j}\}_{j \ne i}) \\ s_{i} &:= \lcm(\{d_{j}\}_{j \ne i}) \\ s &:= \lcm(s_{0},\dotsc,s_{n}) \end{align*} and let $R' := \Z[t_{0}',\dotsc,t_{n}']$ be the ring with the $\Z$-grading determined by $\deg(t_{i}') = \rho_{i}' := \rho_{i}/s_{i}$. The ring homomorphism $R' \to R$ sending $t_{i}' \mapsto t_{i}^{d_{i}}$ (which multiplies the degree by $s$) induces an isomorphism \[ \varphi : \P_{\Z}(\rho) \to \P_{\Z}(\rho') \] of schemes. We have \begin{align} \label{0005-01} \lcm(\rho) = s \cdot \lcm(\rho') \end{align} since $v_{p}(\lcm(\rho)) = \alpha_{i_{0}}$ and $v_{p}(\lcm(\rho')) = \alpha_{i_{0}}-\alpha_{i_{n-1}}$ for any prime $p$, in the notation of \cite[1.2]{ALAMRANI-CDIEGDPDFPT}. \par For any integer $\ell$, there exists a unique pair $(b_{i}(\ell),c_{i}(\ell)) \in \Z^{2}$ satisfying $0 \le b_{i}(\ell) < d_{i}$ and $\ell = b_{i}(\ell) \rho_{i} + c_{i}(\ell)d_{i}$; set $\ell' := \ell-\sum_{i=0}^{n} b_{i}(\ell)\rho_{i}$. The multiplication-by-$(t_{0}^{b_{0}(\ell)} \dotsb t_{n}^{b_{n}(\ell)})$ map $R(\ell') \to R(\ell)$ induces an isomorphism $\mc{O}_{\P_{\Z}(\rho)}(\ell') \simeq \mc{O}_{\P_{\Z}(\rho)}(\ell)$ of $\mc{O}_{\P_{\Z}(\rho)}$-modules. Furthermore $\ell'$ is divisible by $s$ and we obtain an isomorphism \[ \textstyle \mc{O}_{\P_{\Z}(\rho')}(\ell'/s) \simeq \varphi_{\ast}(\mc{O}_{\P_{\Z}(\rho)}(\ell)) \] of $\mc{O}_{\P_{\Z}(\rho)}$-modules. In particular, we have \begin{align} \label{0005-02} \varphi^{\ast}(\mc{O}_{\P_{\Z}(\rho')}(\ell)) \simeq \mc{O}_{\P_{\Z}(\rho)}(s\ell) \end{align} for all $\ell \in \Z$ since $b_{i}(s\ell) = 0$. \end{lemma}

\begin{pg} \label{0023} By \Cref{0005}, all weighted projective lines $\P_{\Z}(q_{0},q_{1})$ are isomorphic to $\P_{\Z}^{1}$; thus, for \Cref{0001}, we may assume $n \ge 2$. \end{pg}

\begin{definition} \cite[\S2]{ALAMRANI-CDIEGDPDFPT} We say that $\rho$ satisfies (N) if $\gcd(\{\rho_{j}\}_{j \ne i}) = 1$ for all $i$. \end{definition}

\begin{pg} \label{0024} \cite[\S8]{ALAMRANI-CDIEGDPDFPT} Let $\rho,\sigma$ be two weight vectors satisfying (N). We have $\P_{\Z}(\rho) \simeq \P_{\Z}(\sigma)$ if and only if $\rho$ is a permutation of $\sigma$. \end{pg}

\begin{lemma} \label{0012} The sheaf $\mc{O}_{\P_{\Z}(\rho)}(r)$ is reflexive for any $r \in \Z$. If $\rho$ satisfies (N), the sheaf $\mc{O}_{\P_{\Z}(\rho)}(r)$ is invertible if and only if $\lcm(\rho)$ divides $r$. \end{lemma}

\begin{lemma}[Picard group of $\P(\rho)$] \label{0002} \cite[\S6.1]{ALAMRANI-CDIEGDPDFPT} For any connected locally Noetherian scheme $S$, the map \[ \Z \oplus \Pic(S) \to \Pic(\P_{S}(\rho)) \] sending \[ (\ell,\mc{L}) \mapsto \mc{O}_{\P_{S}(\rho)}(\ell \cdot \lcm(\rho)) \otimes f_{S}^{\ast}\mc{L} \] is an isomorphism. (See also \cite[\S6]{NOOHI-PSOAWPS}.) \end{lemma} \begin{proof} By \Cref{0022}, we may assume $\gcd(\rho) = 1$. In \cite{ALAMRANI-CDIEGDPDFPT} the desired claim is proved assuming that $\rho$ satisfies (N). If $\rho$ does not satisfy (N), then we conclude using \labelcref{0005-01} and \labelcref{0005-02}. \end{proof}

\begin{lemma}[Cohomology of $\mc{O}_{\P(\rho)}(\ell)$] \label{0003} \cite[\S3]{DELORME-ESPACESPROJECTIFSANISOTROPES} Let $A$ be a ring and set $X := \P_{A}(\rho)$. \begin{enumerate} \item For $\ell \ge 0$, the $A$-module $\H^{0}(X,\mc{O}_{X}(\ell))$ is free with basis consisting of monomials $t_{0}^{e_{0}} \dotsb t_{n}^{e_{n}}$ such that $e_{0},\dotsc,e_{n} \in \Z_{\ge 0}$ and $\rho_{0}e_{0} + \dotsb + \rho_{n}e_{n} = \ell$. \item For $\ell < 0$, the $A$-module $\H^{n}(X,\mc{O}_{X}(\ell))$ is free with basis consisting of monomials $t_{0}^{e_{0}} \dotsb t_{n}^{e_{n}}$ such that $e_{0},\dotsc,e_{n} \in \Z_{< 0}$ and $\rho_{0}e_{0} + \dotsb + \rho_{n}e_{n} = \ell$. \item If $(i,\ell) \not\in (\{0\} \times \Z_{\ge 0}) \cup (\{n\} \times \Z_{< 0})$, then $\H^{i}(X,\mc{O}_{X}(\ell)) = 0$. \item For any $A$-module $M$ and any $(i,\ell)$, the canonical map \[ \H^{i}(X,\mc{O}_{X}(\ell)) \otimes_{A} M \to \H^{i}(X,\mc{O}_{X}(\ell) \otimes_{A} M) \] is an isomorphism. \end{enumerate} \end{lemma}

\begin{remark} Since $\P_{\Z}(\rho) \to \Spec \Z$ is projective, if $S$ is quasi-compact and admits an ample line bundle, then so does $\P_{S}(\rho)$; thus $\Br = \Br'$ for $\P_{S}(\rho)$ by Gabber \cite{DEJONG-GABBER} (i.e. the Azumaya Brauer group coincides with its cohomological Brauer group). \end{remark}

\begin{remark} The projection $\P_{\Z}(\rho) \to \Spec \Z$ is a flat morphism of relative dimension $n$, and its geometric fibers are normal. By \cite[1.3.3. (iii)]{DOLGACHEV-WPV1982}, we have that $\P_{S}(\rho) \to S$ is smooth if and only if $\P_{S}(\rho) \simeq \P_{S}^{n}$. If $\rho$ satisfies (N), then \Cref{0024} implies that $\P_{S}(\rho) \simeq \P_{S}^{n}$ if and only if $\rho = (1,\dotsc,1)$. \end{remark}

\section{Over an algebraically closed field} \label{sec-02}

In this section, we prove \Cref{0001} in the case when $S = \Spec k$ for an algebraically closed field $k$.

\begin{pg}[Presentation as a toric variety] \label{0004} We recall from \cite[\S2.2]{FULTON-ITTV1993}, \cite[Example 3.1.17]{CLS-TORIC} how to view a weighted projective space as a toric variety (i.e. what the fan is). \par Let $\mr{U} \in \GL_{n+1}(\Z)$ be an invertible matrix which has $\rho$ as its 1st row (using the Euclidean algorithm, do column operations on $\rho$ to reduce to $(1,0,\dotsc,0)$, then apply the inverse column operations in the reverse order on the identity matrix $\id_{n+1}$); let $\mr{Y} \in \Mat_{(n+1) \times n}(\Z)$ be the matrix obtained by removing the leftmost column of $\mr{U}^{-1}$; let $\mb{v}_{0},\dotsc,\mb{v}_{n} \in \Z^{n}$ be the rows of $\mr{Y}$; then $\P(\rho)$ is isomorphic to the toric variety associated to the fan $\Delta$ whose maximal cones are generated by the $n$-element subsets of $\{\mb{v}_{0},\dotsc,\mb{v}_{n}\}$. \end{pg}

\begin{pg}[Reduce to computing the subgroup of Zariski-locally trivial Brauer classes] \label{0007} Let $\Delta' \to \Delta$ be a nonsingular subdivision of $\Delta$, and let $X'$ be the toric variety associated to $\Delta'$. The map $X' \to X$ is a toric resolution of singularities for $X$. As in \cite{DEMEYERFORDMIRANDA-TCBGOATV}, we set $\H^{2}(K/X_{\et},\G_{m}) := \ker(\H_{\et}^{2}(X,\G_{m}) \to \H_{\et}^{2}(K,\G_{m}))$; since $X'$ is regular, the restriction $\H_{\et}^{2}(X',\G_{m}) \to \H_{\et}^{2}(K,\G_{m})$ is injective; hence there is an exact sequence \[ 0 \to \H^{2}(K/X_{\et},\G_{m}) \to \H_{\et}^{2}(X,\G_{m}) \to \H_{\et}^{2}(X',\G_{m}) \] of abelian groups. Here $X'$ is a smooth, proper, geometrically connected, rational $k$-scheme; hence $\H_{\et}^{2}(X',\G_{m}) = 0$ by birational invariance of the Brauer group (see \cite[Corollaire (7.3)]{Gro68c} in characteristic $0$ and \cite[Corollary 5.2.6]{COLLIOTTHELENESKOROBOGATOV-TBGG2019} in general); thus it remains to compute $\H^{2}(K/X_{\et},\G_{m})$. By \cite[4.3, 5.1]{DEMEYERFORDMIRANDA-TCBGOATV}, there are natural isomorphisms \begin{align} \label{0007-01} \check{\H}^{2}(\mf{U},\G_{m}) \simeq \H_{\zar}^{2}(X,\G_{m}) \simeq \H^{2}(K/X_{\et},\G_{m}) \end{align} where $\mf{U} = \{U_{\sigma_{0}},\dotsc,U_{\sigma_{n}}\}$ is the Zariski cover of $X$ corresponding to the set of maximal cones of $\Delta$. \end{pg}

\begin{pg}[Limit of dilations] \label{0008} Let $A$ be a ring and let $X$ be the toric variety (over $A$) associated to a fan $\Delta$ of cones in $\ml{N}_{\Q}$. For any positive integer $d$, the multiplication-by-$d$ map $\times d : \ml{N} \to \ml{N}$ induces a finite $A$-morphism \[ \theta_{d} : X \to X \] which is equivariant for the $d$th power map on tori. This is called a \emph{dilation} \cite[\S6]{CORTINASHAESEMEYERWALKERWEIBEL-TKTOTV2009} (or \emph{toric Frobenius} \cite[Remark 4.14]{HERINGMUSTATAPAYNE-PFTVB2008}). For a cone $\sigma$ of $\Delta$, this is the $A$-algebra endomorphism of $\Gamma(U_{\sigma},\mc{O}_{U_{\sigma}}) = A[\sigma^{\vee} \cap \ml{M}]$ sending $\chi^{\ml{m}} \mapsto \chi^{d\ml{m}}$ for $\ml{m} \in \sigma^{\vee} \cap \ml{M}$. If the fan $\Delta$ is smooth, then $\theta_{d}$ is flat for any $d$. \par We view $\N$ as a category whose objects correspond to positive integers $m \in \N$ and there is a morphism $m_{1} \to m_{2}$ if $m_{1}$ divides $m_{2}$. Let $S \subset \N$ be a multiplicatively closed subset; there is a functor $S^{\op} \to (\Sch)$ sending $m \mapsto X$ and $\{m_{1} \to m_{2}\} \mapsto \theta_{m_{2}/m_{1}}$; the limit \[ \textstyle X^{1/S} := \varprojlim (\theta_{m_{2}/m_{1}} : X \to X) \] of the resulting projective system is representable by a scheme since all the transition maps are affine. The scheme $X^{1/S}$ is isomorphic to the monoid scheme obtained by the usual construction with the finite free $\Z$-module $\ml{N}$ and its dual $\ml{M}$ replaced by the $S^{-1}\Z$-module $S^{-1}\ml{N}$ and its dual $S^{-1}\ml{M} = \Hom_{S^{-1}\Z}(S^{-1}\ml{N},S^{-1}\Z)$. More precisely, set $U_{\sigma}^{1/S} := \Spec A[\sigma^{\vee} \cap S^{-1}\ml{M}]$; for any face $\tau$ of $\sigma$, the canonical map $U_{\tau}^{1/S} \to U_{\sigma}^{1/S}$ is an open immersion; then $U_{\sigma_{1}}^{1/S}$ and $U_{\sigma_{2}}^{1/S}$ are glued along the common open subscheme $U_{\sigma_{1} \cap \sigma_{2}}^{1/S}$. \par If $A$ is reduced, then we have \begin{align} \label{0008-02} \Gamma(U_{\sigma},\G_{m}) = (A[\sigma^{\vee} \cap \ml{M}])^{\times} = A^{\times} \cdot (\sigma^{\perp} \cap \ml{M}) \end{align} for any cone $\sigma \in \Delta$; hence, by \labelcref{0007-01}, the pullback \[ \theta_{d}^{\ast} : \H_{\zar}^{p}(X,\G_{m}) \to \H_{\zar}^{p}(X,\G_{m}) \] is multiplication-by-$d$. In the limit, we obtain a natural isomorphism \begin{align} \label{0008-01} \textstyle S^{-1}(\H_{\zar}^{p}(X,\G_{m})) \simeq \H_{\zar}^{p}(X^{1/S},\G_{m}) \end{align} of $S^{-1}\Z$-modules. \end{pg}

\begin{lemma} \label{0010} Let $d$ be a positive integer dividing $\rho_{i}$, and set $\rho' := (\rho_{0}',\dotsc,\rho_{n}')$ where $\rho_{i}' := \rho_{i}/d$ and $\rho_{j}' := \rho_{j}$ for $j \ne i$. If $d \in S$, then $\P_{\Z}(\rho)^{1/S} \simeq \P_{\Z}(\rho')^{1/S}$. \end{lemma} \begin{proof} As in \Cref{0004}, let $\mr{U},\mr{U}' \in \GL_{n+1}(\Z)$ be invertible matrices whose first rows are $\rho,\rho'$ respectively. Let $\mr{U}^{\circ} \in \GL_{n+1}(S^{-1}\Z)$ be the matrix obtained by dividing the $i$th column of $\mr{U}$ by $d$; then $(\mr{U}^{\circ})^{-1}$ is obtained by multiplying the $i$th row of $\mr{U}^{-1}$ by $d$; this does not change the cones since we are just replacing $\mr{v}_{i}'$ by $\frac{1}{d}\mr{v}_{i}'$. Set $\mr{V} := \mr{U}' \cdot (\mr{U}^{\circ})^{-1} \in \GL_{n+1}(S^{-1}\Z)$; since the first rows of $\mr{U}^{\circ},\mr{U}'$ are the same, the matrix $\mr{V}$ has the form \[ \mr{V} = \BM{ 1 & \mb{0} \\ \mr{V}' & \mr{V}''} \] for some $\mr{V}' \in \Mat_{n \times 1}(S^{-1}\Z)$ and $\mr{V}'' \in \GL_{n}(S^{-1}\Z)$. Let $\mr{Y}^{\circ},\mr{Y}' \in \Mat_{(n+1) \times n}(S^{-1}\Z)$ be the matrices obtained by removing the leftmost column of $(\mr{U}^{\circ})^{-1},(\mr{U}')^{-1}$ respectively; then $(\mr{U}')^{-1} \cdot \mr{V} = (\mr{U}^{\circ})^{-1}$ implies $\mr{Y}' \cdot \mr{V}'' = \mr{Y}^{\circ}$; then $\mr{V}'' : S^{-1}\ml{N} \to S^{-1}\ml{N}$ induces the desired isomorphism $\P_{\Z}(\rho)^{1/S} \to \P_{\Z}(\rho')^{1/S}$. \end{proof}

\begin{pg} \label{0011} We show that $\H_{\zar}^{2}(X,\G_{m}) = 0$ by showing that the localizations $\H_{\zar}^{2}(X,\G_{m}) \otimes_{\Z} \Z_{(p)}$ are $0$ for every prime $p$. By \labelcref{0008-01} and \Cref{0010}, we may thus assume that \[ \rho = (1,p^{e_{1}},\dotsc,p^{e_{n}}) \] for some nonnegative integers $e_{1} \le \dotsb \le e_{n}$. In this case, in \Cref{0004} we may take the 1st row of $\mr{U}$ to be $\rho$ and the other rows to coincide with the identity $\id_{n+1}$, so that \begin{align} \label{0011-01} \mr{Y} = \BM{-p^{e_{1}} & \dotsb & -p^{e_{n}} \\ 1 & & \\ & \ddots & \\ & & 1} \end{align} and thus $\mb{v}_{0} = (-p^{e_{1}},\dotsc,-p^{e_{n}})$ and $\mb{v}_{i}$ is the $i$th standard basis vector of $\Z^{n}$. \end{pg}

\begin{pg}[Definition of $\mathsf{A}^{\bullet,\bullet}$] For convenience, we set $[n] := \{0,1,\dotsc,n\}$; we will use $I$ to denote a subset of $[n]$. We construct a double complex \[ (\{\mathsf{A}^{p,q}\},\{\mathsf{d}_{\mr{v}}^{p,q} : \mathsf{A}^{p,q} \to \mathsf{A}^{p,q+1}\},\{\mathsf{d}_{\mr{h}}^{p,q} : \mathsf{A}^{p,q} \to \mathsf{A}^{p+1,q}\}) \] as follows: for $-1 \le p \le n$, we set \begin{equation} \label{20171113-41-eqn03} \begin{aligned} \mathsf{A}^{p,1} &= \textstyle \bigoplus_{|I| = n-p} \Z^{n-p} \\ \mathsf{A}^{p,0} &= \textstyle \bigoplus_{|I| = n-p} \Z^{n} \end{aligned} \end{equation} and $\mathsf{A}^{p,q} = 0$ if $(p,q) \not\in \{-1,\dotsc,n\} \times \{0,1\}$. \par For the vertical differential $\mathsf{d}_{\mr{v}}^{p,0} : \mathsf{A}^{p,0} \to \mathsf{A}^{p,1}$, the $I$th component (with $|I| = n-p$) of this map is the group homomorphism $\Z^{n} \to \Z^{n-p}$ whose corresponding matrix has rows $\mb{v}_{i}$ for $i \in I$. \par The horizontal differentials $\mathsf{d}_{\mr{h}}^{p,q}$ are defined with the sign conventions as follows: if $I = \{i_{0},\dotsc,i_{n-p-1}\} \subset [n]$ is a subset of size $|I| = n-p$ and $I'$ is obtained by removing the $i$th element of $I$ (where $0 \le i \le n-p-1$), then the restriction from the $I$th to $I'$th components has sign $(-1)^{i}$. \par The subcomplex of $A^{\bullet,\bullet}$ obtained by restricting to $p \ge 1$ is isomorphic to the morphism of Cech complexes $\check{\mr{C}}^{\bullet}(\Delta,\mc{M}) \to \check{\mr{C}}^{\bullet}(\Delta,\mc{SF})$, in the notation of \cite[(5.0.1)]{DEMEYERFORDMIRANDA-TCBGOATV}. \end{pg}

\begin{pg}[Diagram of $\mathsf{A}^{\bullet,\bullet}$] Here is a diagram of the double complex $\mathsf{A}^{\bullet,\bullet}$: \begin{center} \begin{tikzpicture}[>=angle 90] 
\matrix[matrix of math nodes,row sep=3em, column sep={5em,between origins}, text height=2ex, text depth=0.0ex] { 
|[name=11]| \mathsf{A}^{-1,1} & |[name=12]| \mathsf{A}^{0,1} & |[name=13]| \mathsf{A}^{1,1} & |[name=14]| \mathsf{A}^{2,1} & |[name=15]| \dotsb & |[name=16]| \mathsf{A}^{n-1,1} & |[name=17]| \mathsf{A}^{n,1} \\ 
|[name=21]| \mathsf{A}^{-1,0} & |[name=22]| \mathsf{A}^{0,0} & |[name=23]| \mathsf{A}^{1,0} & |[name=24]| \mathsf{A}^{2,0} & |[name=25]| \dotsb & |[name=26]| \mathsf{A}^{n-1,0} & |[name=27]| \mathsf{A}^{n,0} \\ 
}; 
\draw[->,font=\scriptsize] (11) edge node[above=0pt] {$\mathsf{d}_{\mr{h}}^{-1,1}$} (12) (12) edge node[above=0pt] {$\mathsf{d}_{\mr{h}}^{0,1}$} (13) (13) edge node[above=0pt] {$\mathsf{d}_{\mr{h}}^{1,1}$} (14) (14) edge (15) (15) edge (16) (16) edge node[above=0pt] {$\mathsf{d}_{\mr{h}}^{n-1,1}$} (17) (21) edge node[below=0pt] {$\mathsf{d}_{\mr{h}}^{-1,1}$} (22) (22) edge node[below=0pt] {$\mathsf{d}_{\mr{h}}^{0,1}$} (23) (23) edge node[below=0pt] {$\mathsf{d}_{\mr{h}}^{1,1}$} (24) (24) edge (25) (25) edge (26) (26) edge node[below=0pt] {$\mathsf{d}_{\mr{h}}^{n-1,0}$} (27) (21) edge node[left=0pt] {$\mathsf{d}_{\mr{v}}^{-1,0}$} (11) (22) edge node[left=0pt] {$\mathsf{d}_{\mr{v}}^{0,0}$} (12) (23) edge node[left=0pt] {$\mathsf{d}_{\mr{v}}^{1,0}$} (13) (24) edge node[left=0pt] {$\mathsf{d}_{\mr{v}}^{2,0}$} (14) (26) edge node[left=0pt] {$\mathsf{d}_{\mr{v}}^{n-1,0}$} (16) (27) edge node[left=0pt] {$\mathsf{d}_{\mr{v}}^{n,0}$} (17); \end{tikzpicture} \end{center} For a weighted projective surface (i.e. $n=2$), this looks like \begin{center} \begin{tikzpicture}[>=angle 90] 
\matrix[matrix of math nodes,row sep=6em, column sep={9em,between origins}, text height=1.7ex, text depth=0.5ex] { 
|[name=11]| \Z^{3} & |[name=12]| \Z^{2} \oplus \Z^{2} \oplus \Z^{2} & |[name=13]| \Z^{1} \oplus \Z^{1} \oplus \Z^{1} & |[name=14]| 0 \\ 
|[name=21]| \Z^{2} & |[name=22]| \Z^{2} \oplus \Z^{2} \oplus \Z^{2} & |[name=23]| \Z^{2} \oplus \Z^{2} \oplus \Z^{2} & |[name=24]| \Z^{2} \\ 
}; 
\draw[->,font=\scriptsize] (21) edge node[fill=white] {$\begin{bmatrix} \mb{v}_{0} \\ \mb{v}_{1} \\ \mb{v}_{2} \end{bmatrix}$} (11) (22) edge node[fill=white] {$\begin{bmatrix} \mb{v}_{0} \\ \mb{v}_{1} \end{bmatrix} \begin{bmatrix} \mb{v}_{0} \\ \mb{v}_{2} \end{bmatrix} \begin{bmatrix} \mb{v}_{1} \\ \mb{v}_{2} \end{bmatrix}$} (12) (23) edge node[fill=white] {$\begin{bmatrix} \mb{v}_{0} \end{bmatrix} \begin{bmatrix} \mb{v}_{1} \end{bmatrix} \begin{bmatrix} \mb{v}_{2} \end{bmatrix}$} (13) (24) edge (14) (11) edge (12) (12) edge (13) (13) edge (14) (21) edge (22) (22) edge (23) (23) edge (24); \end{tikzpicture} \end{center} \end{pg}

\begin{pg} Let $\mb{C}_{n}^{\bullet}$ be the complex with $\mb{C}_{n}^{k} = \Z^{\binom{n}{k}}$ and such that the differentials $\mb{C}_{n}^{k} \to \mb{C}_{n}^{k+1}$ have sign conventions as above. Then $\mb{C}_{n}^{\bullet}$ is isomorphic to a direct sum of shifts of $\id : \Z^{n} \to \Z^{n}$, hence is exact. The complex $\mathsf{A}^{\bullet,0}$ is isomorphic to the direct sum $(\mb{C}_{n+1}^{\bullet})^{n}$, hence is exact. The complex $\mathsf{A}^{\bullet,1}$ is isomorphic to the direct sum $(\mb{C}_{n-1}^{\bullet})^{n+1}$, hence is exact. Let \[ (\{\mr{E}_{r}^{p,q}\},\{\mr{d}_{r}^{p,q} : \mr{E}_{r}^{p,q} \to \mr{E}_{r}^{p+r,q-r+1}\}) \] denote the spectral sequence corresponding to the horizontal filtration on $\mathsf{A}^{\bullet,\bullet}$, so that $\mr{E}_{0}^{p,q} = \mathsf{A}^{p,q}$ and $\mr{d}_{0}^{p,q} = \mathsf{d}_{\mr{v}}^{p,q}$. Then $\mr{E}_{2}^{p,0}$ is identified with our Cech cohomology groups $\check{\H}^{p}(\mf{U},\G_{m})$, where $\mf{U}$ is the Zariski open cover of $X$ corresponding to the maximal cones of $\Delta$. Since there are only two nonzero rows, the differentials \[ \mr{d}_{2}^{p,1} : \mr{E}_{2}^{p,1} \to \mr{E}_{2}^{p+2,0} \] are isomorphisms for all $p$. We are interested in $\check{\H}^{2}(\mf{U},\G_{m}) \simeq \mr{E}_{2}^{2,0} \simeq \mr{E}_{2}^{0,1}$. \end{pg}

\begin{pg} For the differential $\mathsf{d}_{\mr{v}}^{0,0} : \mathsf{A}^{0,0} \to \mathsf{A}^{0,1}$, the $I$th component (with $|I| = n$) of this map is the group homomorphism $\Z^{n} \to \Z^{n}$ whose corresponding matrix is obtained by removing the $i$th rows from $\mr{Y}$ \labelcref{0011-01} for $i \not\in I$; hence \begin{align} \label{20171113-41-eqn01} \textstyle \mr{E}_{1}^{0,1} \simeq \bigoplus_{i \in [n]} \Z/(p^{e_{i}}) \end{align} where a generator of the $i$th component $\Z/(p^{e_{i}})$ is given by the image of the $1$st standard basis vector of $\Z^{n}$ (see \labelcref{20171113-41-eqn03}). \par For the differential $\mathsf{d}_{\mr{v}}^{1,0} : \mathsf{A}^{1,0} \to \mathsf{A}^{1,1}$, the $I$th component (with $|I| = n-1$) of this map is the group homomorphism $\Z^{n} \to \Z^{n-1}$ whose corresponding matrix is obtained by removing the $i$th rows from $\mr{Y}$ \labelcref{0011-01} for $i \not\in I$; hence \begin{align} \label{20171113-41-eqn02} \textstyle \mr{E}_{1}^{1,1} \simeq \bigoplus_{i_{1} < i_{2}} \Z/(p^{\min\{e_{i_{1}},e_{i_{2}}\}}) \end{align} where a generator of the $i$th component $\Z/(p^{\min\{e_{i_{1}},e_{i_{2}}\}})$ is given by the image of the $1$st standard basis vector of $\Z^{n-1}$ (see \labelcref{20171113-41-eqn03}). \end{pg}

\begin{pg} We compute $\mr{E}_{2}^{0,1} = \ker\mr{d}_{1}^{0,1}/\im\mr{d}_{1}^{-1,1}$. With identifications as in \labelcref{20171113-41-eqn01} and \labelcref{20171113-41-eqn02}, the image of $(x_{0},x_{1},\dotsc,x_{n}) \in \mr{E}_{1}^{0,1}$ under the differential $\mr{d}_{1}^{0,1} : \mr{E}_{1}^{0,1} \to \mr{E}_{1}^{1,1}$ has $(i_{1},i_{2})$th coordinate $(-1)^{i_{1}}x_{i_{1}}+(-1)^{i_{2}-1}x_{i_{2}}$. Suppose $(x_{0},x_{1},\dotsc,x_{n}) \in \ker\mr{d}_{1}^{0,1}$; using the differential $\mr{d}_{1}^{-1,1} : \mr{E}_{1}^{-1,1} \to \mr{E}_{1}^{0,1}$, we may assume that $x_{n} = 0$ in $\Z/(p^{e_{n}})$. Since $e_{n-1} \le e_{n}$, the condition $(-1)^{n-1}x_{n-1} + (-1)^{n-1}x_{n} = 0$ in $\Z/(p^{e_{n-1}})$ forces $x_{n-1} = 0$ in $\Z/(p^{e_{n-1}})$. Using downward induction on $i$, we conclude that $x_{i} = 0$ in $\Z/(p^{e_{i}})$ for all $i$. Thus $\mr{E}_{2}^{0,1} = 0$. \end{pg}

\begin{remark}[Assumptions on the base field] In \cite{DEMEYERFORDMIRANDA-TCBGOATV}, it seems there are two implicit assumptions regarding the base field $k$: \begin{enumerate} \item It is assumed that $k$ is algebraically closed. This is used to conclude that all closed points are $k$-points and to identify the henselization and the strict henselization at a closed point of a variety. In the proof of Lemma 4.1, the reference to \cite[VIII, \S13, Theorem 32]{ZARISKI-SAMUEL} (in showing that an affine toric variety is analytically normal) requires $k$ to be perfect (here we may also use \cite[(33.I) Theorem 79]{MATSUMURA-CA}). \item It is assumed that $k$ has characteristic $0$. This is used to conclude that (5.1.1) is a split surjection; we only use their Lemmas 4.3 and 5.1, which do not depend on the characteristic of $k$. There are potential subtleties when considering the Brauer group of (affine) toric varieties in positive characteristic; for example, if $k$ is an algebraically closed field of characteristic $p$, the Brauer group of $k[x_{1},x_{2}]$ has nontrivial $p$-torsion \cite[7.5]{AUSLANDER-GOLDMAN-BRAUER}, and these are not cup products (since $\H_{\fppf}^{1}(\A_{k}^{2},\mu_{p}) = 0$). \end{enumerate} \end{remark}

\section{Over a general base scheme}

The following lemma \Cref{0006}, combined with \Cref{sec-02}, proves \Cref{0001} when $S = \Spec k$ for an arbitrary field $k$.

\begin{lemma} \label{0006} Let $A$ be a local ring, set $X := \P_{A}(\rho)$, let $P \in X(A)$ be an $A$-rational point and let $\alpha \in \H_{\et}^{2}(X,\G_{m})$ be a class such that $\alpha_{P} = 0$. If there exists a finite faithfully flat $A$-algebra $A'$ such that $\alpha_{A'} = 0$, then $\alpha = 0$. \end{lemma} \begin{proof} Let $\mc{G} \to X$ be the $\G_{m}$-gerbe corresponding to $\alpha$. Since $\mc{G}_{A'}$ is trivial, there is a 1-twisted line bundle $\mc{L}'$ on $\mc{G}_{A'}$; set $A'' := A' \otimes_{A} A'$ and $A''' := A' \otimes_{A} A' \otimes_{A} A'$; then there exists a line bundle $L''$ on $X_{A''}$ such that $L''|_{\mc{G}_{A''}} \simeq (p_{1}^{\ast}\mc{L}')^{-1} \otimes p_{2}^{\ast}\mc{L}'$; this line bundle $L''$ satisfies $p_{13}^{\ast}L'' \simeq p_{23}^{\ast}L'' \otimes p_{12}^{\ast}L''$; hence $L''$ is trivial since $p_{12}^{\ast},p_{13}^{\ast},p_{23}^{\ast} : \Pic(X_{A''}) \to \Pic(X_{A'''})$ are the same maps $\Z \to \Z$ (see \Cref{0002}). Choose an isomorphism $\varphi : p_{1}^{\ast}\mc{L}' \to p_{2}^{\ast}\mc{L}'$ of $\mc{O}_{\mc{G}_{A''}}$-modules; the isomorphisms $p_{13}^{\ast}\varphi$ and $p_{23}^{\ast}\varphi \circ p_{12}^{\ast}\varphi$ differ by an element $u_{\alpha} \in \Gamma(X_{A'''},\G_{m}) \simeq \Gamma(A''',\G_{m})$. Since $\mc{G}|_{P}$ is trivial, we may refine the finite flat cover $A \to A'$ if necessary so that $u_{\alpha}$ is the coboundary of some $u_{\beta} \in \Gamma(X_{A''},\G_{m})$. After modifying $\varphi$ by this $u_{\beta}$, we have that the descent datum $(\mc{L}',\varphi)$ gives a 1-twisted line bundle on $\mc{G}$. \end{proof}

We use deformation theory of twisted sheaves to deduce \Cref{0001} over strictly henselian local rings:

\begin{lemma} \label{20190128-05} Let $A$ be a strictly henselian local ring. Then $\H_{\et}^{2}(\P_{A}(\rho),\G_{m}) = 0$. \end{lemma} \begin{proof} \footnote{This is an argument of Siddharth Mathur \cite{MATHUR-EC20190726}.} By standard limit techniques, we may assume that $A$ is the strict henselization of a localization of a finite type $\Z$-algebra; in particular, $A$ is excellent \cite[5.6 iii)]{GRECO-TTOER1976}. Let $\mf{m}$ be the maximal ideal of $A$ and let $k := A/\mf{m}$ be the residue field. \par We first consider the case when $A$ is complete. Set $X := \P_{A}(\rho)$ and let $\pi : \mc{G} \to X$ be a $\G_{m}$-gerbe corresponding to a class $[\mc{G}] \in \H_{\et}^{2}(X,\G_{m})$. The class $[\mc{G}]$ is trivial if and only if $\pi$ admits a section. We have that $\mc{G}_{0}$ is a $\G_{m}$-gerbe over $X_{0} = \P_{k}(\rho)$, which is a trivial gerbe by \Cref{0006} since $k$ is separably closed. For $\ell \in \N$, set $X_{\ell} := X \times_{\Spec A} \Spec A/\mf{m}^{\ell+1}$ and $\mc{G}_{\ell} := \mc{G} \times_{X} X_{\ell}$. We have equivalences of categories \begin{align*} \Mor(X,\mc{G}) &\stackrel{1}{\simeq} \Hom_{r\otimes,\simeq}(\mathsf{Coh}(\mc{G}),\mathsf{Coh}(X)) \\ &\stackrel{2}{\simeq} \Hom_{r\otimes,\simeq}(\mathsf{Coh}(\mc{G}),\varprojlim\mathsf{Coh}(X_{\ell})) \\ &\stackrel{3}{\simeq} \varprojlim \Hom_{r\otimes,\simeq}(\mathsf{Coh}(\mc{G}),\mathsf{Coh}(X_{\ell})) \\ &\stackrel{1}{\simeq} \varprojlim \Mor(X_{\ell},\mc{G}) \end{align*} where the equivalences marked 1 are by \cite[Theorem 1.1]{HALLRYDH-CTDAAOHS2014} (here we use that $A$ is excellent), the equivalence marked 2 is Grothendieck existence \cite[$\text{III}_{1}$, 5.1.4]{EGA}, the equivalence marked 3 is \cite[Lemma 3.8]{HALLRYDH-CTDAAOHS2014}. \par It remains now to construct a compatible system of morphisms $X_{\ell} \to \mc{G}$. A morphism $X_{\ell} \to \mc{G}$ over $\P_{A}(\rho)$ corresponds to a 1-twisted line bundle on $\mc{G}_{\ell}$; the obstruction to lifting a line bundle via $\mc{G}_{\ell} \to \mc{G}_{\ell+1}$ lies in $\H_{\et}^{2}(\mc{G}_{\ell} , \mf{m}^{\ell}\mc{O}_{\mc{G}_{\ell}})$; by \Cref{20160903-37}, this is isomorphic to $\H_{\et}^{2}(X_{\ell},\mf{m}^{\ell}\mc{O}_{X_{\ell}})$, which is $0$ by \Cref{0003}. \par In general, if $A$ is not complete, we use Artin approximation to descend a 1-twisted line bundle from $\mc{G}^{\wedge}$ to $\mc{G}$. \end{proof}

\begin{lemma} \label{20160903-37} Let $X$ be an algebraic stack, let $\mr{A}$ be a finitely generated abelian group, let $\mb{G} = \mr{D}(\mr{A})$ be the diagonalizable group scheme associated to $\mr{A}$, and let $\pi : \mc{G} \to X$ be a $\mb{G}$-gerbe. For any quasi-coherent $\mc{O}_{X}$-module $F$, the pullback maps \begin{align} \label{20160903-37-eqn-01} \H_{\et}^{i}(X,F) \to \H_{\et}^{i}(\mc{G},\pi^{\ast}F) \end{align} are isomorphisms for all $i$. \end{lemma} \begin{proof} Set $\ms{F} := \pi^{\ast}F$. We first assume that $X$ is a scheme. We have a Leray spectral sequence \[ \mr{E}_{2}^{p,q} = \H_{\et}^{p}(X,\mb{R}^{q}\pi_{\ast}(\ms{F})) \Rightarrow \H_{\et}^{p+q}(\mc{G},\ms{F}) \] with differentials $\mr{E}_{2}^{p,q} \to \mr{E}_{2}^{p+2,q-1}$. Since $\pi_{\ast}\ms{F} \simeq F$, it suffices to show that $\mb{R}^{q}\pi_{\ast}(\ms{F}) = 0$ for $q \ge 1$. Here the stalks of $\mb{R}^{q}\pi_{\ast}(\ms{F})$ are the cohomology $\H_{\et}^{q}(\mr{B}\mb{G}_{A} , \ms{F}|_{A})$ for strictly henselian local rings $A := \mc{O}_{X,x}^{\mr{sh}}$. Set $\mc{G} := \mr{B}\mb{G}_{A}$; we show that if $\ms{F}$ is any quasi-coherent $\mc{O}_{\mc{G}}$-module then $\H^{p}_{\et}(\mc{G},\ms{F}) = 0$ for all $p > 0$. The category of quasi-coherent $\mc{O}_{\mc{G}}$-modules corresponds to the category $\mc{C}$ of $\Z$-graded $A$-modules. Denoting by $\pi : \mc{G} \to \Spec A$ the structure map, the pushforward functor $\pi_{\ast} : \mr{QCoh}(\mc{G}) \to \mr{QCoh}(A)$ corresponds to sending a $\Z$-graded module $M_{\bullet} = \bigoplus_{n \in \Z} M_{n}$ to the degree zero component $M_{0}$. Since this is an exact functor, we have that $\pi$ is cohomologically affine \cite[Definition 3.1]{ALPER-GMSFAS}. Since $\pi$ has affine diagonal, we have the desired result by \cite[Remark 3.5]{ALPER-GMSFAS}. \par In case $X$ is an algebraic space, let $U \to X$ be an etale surjection where $U$ is a scheme, and let $U^{p} := U \times_{X} \dotsb \times_{X} U$ be the $(p+1)$-fold fiber product; then we have descent spectral sequences \begin{align*} \mr{E}_{2}^{p,q} &= \H_{\et}^{q}(U^{p},F|_{U^{p}}) \Rightarrow \H_{\et}^{p+q}(X,F) \\ \mr{E}_{2}^{p,q} &= \H_{\et}^{q}(\mc{G}_{U^{p}},F|_{\mc{G}_{U^{p}}}) \Rightarrow \H_{\et}^{p+q}(\mc{G},\ms{F}) \end{align*} where, by the first paragraph, the pullback \[ \H_{\et}^{q}(U^{p},F|_{U^{p}}) \to \H_{\et}^{q}(\mc{G}_{U^{p}},F|_{\mc{G}_{U^{p}}}) \] is an isomorphism for all $p,q$ since each $U^{p}$ is a scheme; hence \labelcref{20160903-37-eqn-01} is an isomorphism. In case $X$ is an algebraic stack, we take $U \to X$ to be a smooth surjection where $U$ is a scheme and argue as above, noting that each $U^{p}$ is an algebraic space. \end{proof}

\begin{pg}[Proof of {\Cref{0001}}] \label{0009} Set $f := f_{X}$. The Leray spectral sequence associated to the map $f$ and sheaf $\G_{m}$ is of the form \begin{align} \label{sec-05-eqn-01} \mr{E}_{2}^{p,q} = \H^{p}_{\et}(S , \mb{R}^{q}f_{\ast}\G_{m}) \Rightarrow \H^{p+q}_{\et}(X , \G_{m}) \end{align} with differentials $\mr{d}_{2}^{p,q} : \mr{E}_{2}^{p,q} \to \mr{E}_{2}^{p+2,q-1}$. For any strictly henselian local ring $A$, we have $\H^{2}_{\et}(\P_{A}(\rho) , \G_{m}) = 0$ by \Cref{20190128-05}, hence $\mb{R}^{2}f_{\ast}\G_{m} = 0$ since its stalks vanish. The sheaf $\mb{R}^{1}f_{\ast}\G_{m}$ is the sheaf associated to $T \mapsto \Pic(X_{T})$; by \Cref{0002}, every line bundle on $\P_{T}(\rho)$ is, locally on $T$, isomorphic to one pulled back from $\P_{\Z}(\rho)$; hence $\mb{R}^{1}f_{\ast}\G_{m}$ is isomorphic to the constant sheaf $\underline{\Z}$. Hence we have an exact sequence \begin{align} \label{01-eqn-01} \H^{0}_{\et}(S , \underline{\Z}) \stackrel{\dagger}{\to} \H^{2}_{\et}(S , \G_{m}) \stackrel{f^{\ast}}{\to} \H^{2}_{\et}(X , \G_{m}) \to \H^{1}_{\et}(S , \underline{\Z}) \end{align} and we may argue as in \cite[5.4]{SHIN-TCBGOATGMG} to show that $f^{\ast}$ restricts to a surjection on the torsion subgroups, inducing an exact sequence \labelcref{0001-01} as desired. \qed \end{pg}

\begin{remark} From \Cref{0009}, we see that the map $\Gamma(S,\underline{\Z}) \to \Br'(S)$ in \labelcref{0001-01} corresponds to the differential $\mr{d}_{2}^{0,1}$ in the Leray spectral sequence. The Brauer class corresponding to the image of $1 \in \Gamma(S,\underline{\Z})$ may be described as follows. Set $R := \Z[t_{0},\dotsc,t_{n}]$ with $\deg(t_{i}) = \rho_{i}$ and let $\Aut_{\mr{gr.alg.}}(R)$ denote the group sheaf sending a scheme $T$ to the set of $\Z$-graded $\mc{O}_{T}$-algebra automorphisms of $R \otimes_{\Z} \mc{O}_{T}$. By \cite[\S8]{ALAMRANI-CDIEGDPDFPT}, we have an exact sequence \[ 1 \to \G_{m} \to \Aut_{\mr{gr.alg.}}(R) \to \Aut_{\mr{sch}}(\P_{\Z}(\rho)) \to 1 \] of sheaves of groups for the etale topology on the category of schemes, where the image of $\G_{m}$ is contained in the center of $\Aut_{\mr{gr.alg.}}(R)$. By definition, $X$ is an $\Aut_{\mr{sch}}(\P_{\Z}(\rho))$-torsor over $S$, and the class of $[X]$ under the coboundary map \[ \H_{\et}^{1}(S,\Aut_{\mr{sch}}(\P_{\Z}(\rho))) \to \H_{\et}^{2}(S,\G_{m}) \] is the desired Brauer class. \par Alternatively, fix an etale surjection $S' \to S$ and set $S'' := S' \times_{S} S'$ and $S''' := S' \times_{S} S' \times_{S} S'$; the choice of an isomorphism $X \times_{S} S' \simeq \P_{S'}(\rho)$ yields an automorphism $\varphi : \P_{S''}(\rho) \to \P_{S''}(\rho)$ satisfying the cocycle condition $p_{13}^{\ast}\varphi = p_{23}^{\ast}\varphi \circ p_{12}^{\ast}\varphi$ over $S'''$. Choose $\ell \gg 0$ so that $\mc{O}_{\P(\rho)}(\ell)$ is very ample; fixing a $\Z$-basis of $\Gamma(\P_{\Z}(\rho),\mc{O}_{\P_{\Z}(\rho)}(\ell))$ gives an invertible matrix $\varphi^{\sharp} \in \GL_{r}(\Gamma(S'',\mc{O}_{S''}))$, where $r = \rank_{\Z} \Gamma(\P_{\Z}(\rho),\mc{O}_{\P_{\Z}(\rho)}(\ell))$; here the invertible matrices $p_{13}^{\ast}\varphi^{\sharp},p_{12}^{\ast}\varphi^{\sharp} \cdot p_{23}^{\ast}\varphi^{\sharp} \in \GL_{r}(\Gamma(S''',\mc{O}_{S'''}))$ differ by a unit $u \in \Gamma(S''',\G_{m})$, which is the desired class in $\H_{\et}^{2}(S,\G_{m})$. In other words, given a $\Z$-graded algebra automorphism of $R$, it restricts to a $\Z$-graded algebra automorphism of its $\ell$th Veronese subring $R^{(\ell)} := \bigoplus_{i \ge 0} R_{i\ell}$, which restricts to an abelian group automorphism of $R_{\ell}$ and thus a $\Z$-graded algebra automorphism of the standard graded algebra $\Sym_{\Z}^{\bullet}R_{\ell} \simeq \Z[t'_{1},\dotsc,t'_{r}]$; the induced group homomorphism $\Aut_{\mr{gr.alg.}}(R) \to \Aut_{\mr{gr.alg.}}(\Sym_{\Z}^{\bullet}R_{\ell})$ induces a commutative diagram of exact sequences which we may use to compare the two constructions above.  \end{remark}

\begin{remark}[Comparison to the argument of Gabber] In \cite{GABBER-THESIS}, Gabber computes the Brauer group of Brauer-Severi schemes over an arbitrary base scheme by combining the following two facts to reduce to the $\P^{1}$ case: \begin{enumerate} \item Suppose $Y \to X$ is a closed immersion locally defined by a regular sequence, and let $B \to X$ be the blowup of $X$ at $Y$; then $\H_{\et}^{2}(X,\G_{m}) \to \H_{\et}^{2}(B,\G_{m})$ is injective. \item The blowup of $\P^{n}$ at a point is a $\P^{1}$-bundle over $\P^{n-1}$. \end{enumerate} In our case, we may ask whether the analogous statement to (2) holds -- namely, whether a (weighted) blowup of $\P(\rho)$ at a (torus-invariant) local complete intersection subscheme is isomorphic to a $\P(\rho')$-bundle over $\P(\rho'')$ for some $\rho',\rho''$ such that $|\rho|-1 = |\rho'|-1 + |\rho''|-1$. Indeed, the blowup of the weighted projective surface $\P(1,1,q_{2})$ at its unique singular point gives the $q_{2}$th Hirzebruch surface $\F_{q_{2}}$ (see \cite[1.2.3]{DOLGACHEV-WPV1982}, \cite{GAUDUCHON-HSAWPP2009}). Such a result for arbitrary $\rho$ would give an alternative proof of \Cref{0001}. This seems unlikely, however, as it (with \Cref{0023}) would imply that every weighted projective surface $\P(\rho_{0},\rho_{1},\rho_{2})$ is a $\P^{1}$-bundle over $\P^{1}$, which has Picard group $\Z^{2}$, but $\P(2,3,5)$ has three isolated singular points and blowing up these points increases the Picard rank by 3. \end{remark}

\section{Weighted projective stacks}

In this section we assume $n \ge 1$.

The weighted projective stack $\mc{P}_{\Z}(\rho)$ is smooth for any $\rho$ (hence $\mc{P}_{\Z}(\rho) \to \P_{\Z}(\rho)$ is not an isomorphism if $\rho \ne (1,\dotsc,1)$). A Deligne-Mumford stack $\mc{X}$ and its coarse moduli space $X$ may have different Brauer groups (and Picard groups) in general.

\begin{lemma} \label{0013} For any field $k$, the pullback map \[ \H_{\et}^{2}(\Spec k,\G_{m}) \to \H_{\et}^{2}(\mc{P}_{k}(\rho),\G_{m}) \] is an isomorphism. \end{lemma} \begin{proof} We have a descent spectral sequence \begin{align} \label{0013-eqn01} \mr{E}_{1}^{p,q} = \H^{q}_{\et}(\G_{m,k}^{\times p} \times_{k} (\A_{k}^{n+1} \setminus \{0\}) , \G_{m}) \implies \H^{p+q}_{\et}(\mc{P}_{k}(\rho),\G_{m}) \end{align} with differentials $\mr{d}_{1}^{p,q} : \mr{E}_{1}^{p,q} \to \mr{E}_{1}^{p+1,q}$. Each $\G_{m,k}^{\times p} \times_{k} (\A_{k}^{n+1} \setminus \{0\})$ is an open subscheme of $\A_{k}^{n+p+1}$, hence has trivial Picard group; hence $\mr{E}_{1}^{p,1} = 0$ for all $p$. The pullback $\mr{B}\G_{m,k} \to \mc{P}_{k}(\rho)$ induces an isomorphism of complexes $\H_{\et}^{0}(\G_{m,k}^{\times \bullet} , \G_{m}) \to \mr{E}_{1}^{\bullet,0}$; hence, by the proof of \cite[Lemma 4.2]{SHIN-TCBGOATGMG}, we have $\mr{E}_{2}^{2,0} = 0$. \par It remains to compute $\mr{E}_{2}^{0,2}$, which is isomorphic to the equalizer of the two pullback maps \[ a^{\ast},p_{2}^{\ast} : \H_{\et}^{2}(\A_{k}^{n+1} \setminus \{0\} , \G_{m}) \rightrightarrows \H_{\et}^{2}(\G_{m} \times_{k} (\A_{k}^{n+1} \setminus \{0\}),\G_{m}) \] corresponding to the action map and second projection, respectively; by purity for the Brauer group (see Gabber \cite{FUJIWARA-APOTAPC} and {\v C}esnavi{\v c}ius \cite{CESNAVICIUS-PFTBG}), this is isomorphic to the equalizer of $a^{\ast},p_{2}^{\ast} : \H_{\et}^{2}(\A_{k}^{n+1} , \G_{m}) \rightrightarrows \H_{\et}^{2}(\G_{m} \times_{k} \A_{k}^{n+1},\G_{m})$, and also to the equalizer of $a^{\ast},p_{2}^{\ast} : \H_{\et}^{2}(\A_{k}^{n+1} , \G_{m}) \rightrightarrows \H_{\et}^{2}(\A_{k}^{1} \times_{k} \A_{k}^{n+1},\G_{m})$ since the restriction $\H_{\et}^{2}(\A_{k}^{1} \times_{k} \A_{k}^{n+1},\G_{m}) \to \H_{\et}^{2}(\G_{m} \times_{k} \A_{k}^{n+1},\G_{m})$ is injective. With coordinates $\A_{k}^{1} = \Spec k[u]$, let $f : \A_{k}^{1} \times_{k} \A_{k}^{n+1} \to \A_{k}^{n+1}$ be the morphism of $k$-schemes obtained by setting $u = 0$; note that $p_{2}f = \id$ and $af$ factors through $\Spec k$. Let $\alpha \in \H_{\et}^{2}(\A_{k}^{n+1},\G_{m})$ be a Brauer class such that $a^{\ast}\alpha = p_{2}^{\ast}\alpha$ in $\H_{\et}^{2}(\A_{k}^{1} \times_{k} \A_{k}^{n+1},\G_{m})$; then $f^{\ast}a^{\ast}\alpha = f^{\ast}p_{2}^{\ast}\alpha = \alpha$; hence $\alpha$ is in the image of $\H_{\et}^{2}(\Spec k , \G_{m})$. \end{proof}

\begin{lemma} \label{0018} \cite[4.3]{NOOHI-PSOAWPS} For any connected scheme $S$, the map \[ \Z \oplus \Pic(S) \to \Pic(\mc{P}_{S}(\rho)) \] sending \[ (\ell,\mc{L}) \mapsto \mc{O}_{\mc{P}_{S}(\rho)}(\ell) \otimes \pi_{S}^{\ast}\mc{L} \] is an isomorphism. \end{lemma}

\begin{lemma}[Cohomology of $\mc{O}_{\mc{P}(\rho)}(\ell)$] \label{0015} \cite[2.5]{MEIER-VBOTMSOEC} Let $A$ be a ring and set $X := \mc{P}_{A}(\rho)$. \begin{enumerate} \item For $\ell \ge 0$, the $A$-module $\H^{0}(X,\mc{O}_{X}(\ell))$ is free with basis consisting of monomials $t_{0}^{e_{0}} \dotsb t_{n}^{e_{n}}$ such that $e_{0},\dotsc,e_{n} \in \Z_{\ge 0}$ and $\rho_{0}e_{0} + \dotsb + \rho_{n}e_{n} = \ell$. \item For $\ell < 0$, the $A$-module $\H^{n}(X,\mc{O}_{X}(\ell))$ is free with basis consisting of monomials $t_{0}^{e_{0}} \dotsb t_{n}^{e_{n}}$ such that $e_{0},\dotsc,e_{n} \in \Z_{< 0}$ and $\rho_{0}e_{0} + \dotsb + \rho_{n}e_{n} = \ell$. \item If $(i,\ell) \not\in (\{0\} \times \Z_{\ge 0}) \cup (\{n\} \times \Z_{< 0})$, then $\H^{i}(X,\mc{O}_{X}(\ell)) = 0$. \item For any $A$-module $M$ and any $(i,\ell)$, the canonical map \[ \H^{i}(X,\mc{O}_{X}(\ell)) \otimes_{A} M \to \H^{i}(X,\mc{O}_{X}(\ell) \otimes_{A} M) \] is an isomorphism. \end{enumerate} \end{lemma}

\begin{lemma} \label{0019} Let $A$ be a strictly henselian local ring. Then $\H_{\et}^{2}(\mc{P}_{A}(\rho),\G_{m}) = 0$. \end{lemma} \begin{proof} The proof is the same as that of \Cref{20190128-05} with the following modifications: the gerbe $\mc{G}_{0}$ over the special fiber follows from \Cref{0013}; to obtain the equivalence marked 2, we use Grothendieck existence for stacks \cite[1.4]{OLSSON-ONPROPERCOVERINGSOFARTINSTACKS} (using that $\mc{P}(\rho)$ is proper \cite[2.1]{MEIER-VBOTMSOEC}); to conclude that $\H_{\et}^{2}(X_{\ell},\mf{m}^{\ell}\mc{O}_{X_{\ell}}) = 0$, we use \Cref{0015}. \end{proof}

\begin{lemma} \label{0025} Let \[ \pi_{\rho} : \mc{P}_{\Z}(\rho) \to \P_{\Z}(\rho) \] denote the coarse moduli space morphism. For any $\ell \in \Z$, there is a canonical $\mc{O}_{\mc{P}_{\Z}(\rho)}$-linear map \begin{align} \label{0016-eqn01} \pi_{\rho}^{\ast}(\mc{O}_{\P_{\Z}(\rho)}(\ell)) \to \mc{O}_{\mc{P}_{\Z}(\rho)}(\ell) \end{align} which is an isomorphism if $\ell$ is divisible by $\lcm(\rho)$. \end{lemma} \begin{proof} Set $R := \Z[t_{0},\dotsc,t_{n}]$ with the $\Z$-grading determined by $\deg(t_{i}) = \rho_{i}$. The restriction of \labelcref{0016-eqn01} to the open substack $[(\Spec R[t_{i}^{-1}])/\G_{m}]$ corresponds to the graded homomorphism \begin{align} \label{0016-eqn02} R(\ell)[t_{i}^{-1}]_{0} \otimes_{R[t_{i}^{-1}]_{0}} R[t_{i}^{-1}] \to R(\ell)[t_{i}^{-1}] \end{align} of $\Z$-graded $R[t_{i}^{-1}]$-modules; the $m$th component of \labelcref{0016-eqn02} is isomorphic to the $R[t_{i}^{-1}]_{0}$-linear map \begin{align} \label{0016-eqn03} R[t_{i}^{-1}]_{\ell} \otimes_{R[t_{i}^{-1}]_{0}} R[t_{i}^{-1}]_{m} \to R[t_{i}^{-1}]_{\ell+m} \end{align} induced by multiplication. If $\ell$ is divisible by $\rho_{i}$, then the multiplication-by-$t_{i}^{\ell/\rho_{i}}$ map $R[t_{i}^{-1}] \to R[t_{i}^{-1}](\ell)$ is an isomorphism of $\Z$-graded $R[t_{i}^{-1}]$-modules, thus \labelcref{0016-eqn03} is an isomorphism for all $m \in \Z$, in other words the restriction of \labelcref{0016-eqn01} to $[(\Spec R[t_{i}^{-1}])/\G_{m}]$ is an isomorphism.  \end{proof}

\begin{lemma} \label{0016} The pullback $\pi_{\rho}^{\ast} : \Pic(\P_{\Z}(\rho)) \to \Pic(\mc{P}_{\Z}(\rho))$ is multiplication by $\lcm(\rho)$. \end{lemma} \begin{proof} We have that $\Pic(\P_{\Z}(\rho)) \simeq \Z$ is generated by the class of $\mc{O}_{\P_{\Z}(\rho)}(\lcm(\rho))$ and that $\Pic(\mc{P}_{\Z}(\rho)) \simeq \Z$ is generated by the class of $\mc{O}_{\mc{P}_{\Z}(\rho)}(\lcm(1))$ by \Cref{0002} and \Cref{0018}, respectively. We have the desired claim by \Cref{0025}. \end{proof}

\begin{remark} \label{0017} There exist $\rho,\ell$ for which the natural map \labelcref{0016-eqn01} is not an isomorphism. For example, in case $\rho = (1,2)$ and $\ell = 1$, the element $t_{0} \in R[t_{0}^{-1}]_{2}$ is not in the image of the map \labelcref{0016-eqn03} for $m = 1$ and $i = 0$. We have $\mc{O}_{\P(\rho)}(1) \simeq \mc{O}_{\P(\rho)}$, and the pullback \labelcref{0016-eqn01} is multiplication by $t_{1} \in \Gamma(\mc{P}(\rho),\mc{O}_{\mc{P}(\rho)}(1))$; see \Cref{0005} for details. Furthermore, the natural map $\mc{O}_{\P(\rho)}(1) \otimes \mc{O}_{\P(\rho)}(1) \to \mc{O}_{\P(\rho)}(2)$ is not an isomorphism; here \cite[II, 2.5.13]{EGA} does not apply since $R$ is not generated in degree 1. (See also \cite[4.8]{DELORME-ESPACESPROJECTIFSANISOTROPES}, \cite[1.5.3]{DOLGACHEV-WPV1982}.) \end{remark}

\begin{pg}[Proof of {\Cref{0021}}] \label{0020} The proof of the exactness of \labelcref{0021-01} is the same as in \Cref{0009} with the following modifications: to show $\mb{R}^{2}\pi_{\ast}\G_{m} = 0$, we use \Cref{0019}; to show $\mb{R}^{1}\pi_{\ast}\G_{m} \simeq \underline{\Z}$, we use \Cref{0018}. \par An automorphism of the weighted projective stack induces an automorphism of the weighted projective space by universal properties of a coarse moduli space morphism; hence, if $\mc{X}$ is etale-locally isomorphic to $\mc{P}_{\Z}(\rho)$, then $X$ is etale-locally isomorphic to $\P_{\Z}(\rho)$. \par The commutativity of \labelcref{0021-02} follows from \Cref{0016}. \qed \end{pg}

\bibliography{../allbib.bib}
\bibliographystyle{alpha}

\end{document}